\documentclass[12pt,british,3p, sort&compress, times]{elsarticle}
\usepackage[T1]{fontenc}
\usepackage[latin9]{inputenc}
\usepackage{float}
\usepackage{amsmath}
\usepackage{amsthm}
\usepackage{amssymb}
\usepackage{graphicx}

\makeatletter

\providecommand{\tabularnewline}{\\}

\numberwithin{equation}{section}
\numberwithin{figure}{section}
\theoremstyle{definition}
    \ifx\thechapter\undefined
      \newtheorem{defn}{\protect\definitionname}
    \else
      \newtheorem{defn}{\protect\definitionname}[chapter]
    \fi
\theoremstyle{plain}
    \ifx\thechapter\undefined
	    \newtheorem{thm}{\protect\theoremname}
	  \else
      \newtheorem{thm}{\protect\theoremname}[chapter]
    \fi
\theoremstyle{definition}
    \ifx\thechapter\undefined
      \newtheorem{problem}{\protect\problemname}
    \else
      \newtheorem{problem}{\protect\problemname}[chapter]
    \fi
\theoremstyle{plain}
    \ifx\thechapter\undefined
      \newtheorem{conjecture}{\protect\conjecturename}
    \else
      \newtheorem{conjecture}{\protect\conjecturename}[chapter]
    \fi

\@ifundefined{date}{}{\date{}}
\usepackage{bm}
\usepackage[final]{microtype}
\usepackage[unicode=true,bookmarks=true,breaklinks=true,
bookmarksnumbered=false,bookmarksopen=false,pdfborder={0 0 0},
pdfborderstyle={},backref=false,colorlinks=true,citecolor=cyan,
linkcolor=cyan]{hyperref}
\makeatletter
\def\ps@pprintTitle{%
 \let\@oddhead\@empty
 \let\@evenhead\@empty
 \def\@oddfoot{}%
 \let\@evenfoot\@oddfoot}
\makeatother

\makeatother

\usepackage{babel}
\providecommand{\conjecturename}{Conjecture}
\providecommand{\definitionname}{Definition}
\providecommand{\problemname}{Problem}
\providecommand{\theoremname}{Theorem}

\begin{document}

\begin{frontmatter}{}

\title{Polynomials with symmetric zeros}

\author{R. S. Vieira}

\address{São Paulo State University (Unesp), Faculty of Science and Technology,
Department of Mathematics and Computer Science, Presidente Prudente,
SP, Brazil.}

\ead{email: rs.vieira@unesp.br }
\begin{abstract}
Polynomials whose zeros are symmetric either to the real line or to
the unit circle are very important in mathematics and physics. We
can classify them into three main classes: the self-conjugate polynomials,
whose zeros are symmetric to the real line; the self-inversive polynomials,
whose zeros are symmetric to the unit circle; and the self-reciprocal
polynomials, whose zeros are symmetric by an inversion with respect
to the unit circle followed by a reflection in the real line. Real
self-reciprocal polynomials are simultaneously self-conjugate and
self-inversive so that their zeros are symmetric to both the real
line and the unit circle. In this survey, we present a short review
of these polynomials, focusing on the distribution of their zeros. 
\end{abstract}

\end{frontmatter}{}

\section{Introduction}

In this work, we consider the theory of self-conjugate (SC), self-reciprocal
(SR) and self-inversive (SI) polynomials. These are polynomials whose
zeros are symmetric either to the real line $\mathbb{R}$ or to the
unit circle $\mathbb{S}=\left\{ z\in\mathbb{C}:|z|=1\right\} $. The
basic properties of these polynomials can be found in the books of
Marden \cite{Marden1966}, Milovanovi\'c \& al. \cite{Milovanovic1994},
Sheil-Small \cite{Sheil2002}. Although these polynomials are very
important in both mathematics and physics, it seems that there is
no specific review about them; in this work we present a bird's eye
view to this theory, focusing on the zeros of such polynomials. Other
aspects of the theory (e.g., irreducibility, norms, analytical properties
etc.) are not covered here due to the short space, nonetheless, the
interested reader can check many of the references presented in the
bibliography to this end.

\section{Self-conjugate, Self-reciprocal and Self-inversive polynomials\label{SectionBasics}}

We begin with some definitions:
\begin{defn}
Let $p(z)=p_{0}+p_{1}z+\cdots+p_{n-1}z^{n-1}+p_{n}z^{n}$ be a polynomial
of degree $n$ with complex coefficients. We shall introduce three
polynomials --- namely, the \emph{conjugate polynomial} $\overline{p}(z)$,
the \emph{reciprocal polynomial} $p^{\ast}(z)$ and the \emph{inversive
polynomial} $p^{\dagger}(z)$ ---, that are respectively defined
in terms of $p(z)$ as follows:
\begin{align}
\overline{p}(z) & =\overline{p_{0}}+\overline{p_{1}}z+\cdots+\overline{p_{n-1}}z^{n-1}+\overline{p_{n}}z^{n},\nonumber \\
p^{\ast}(z) & =p_{n}+p_{n-1}z+\cdots+p_{1}z^{n-1}+p_{0}z^{n},\nonumber \\
p^{\dagger}(z) & =\overline{p_{n}}+\overline{p_{n-1}}z+\cdots+\overline{p_{1}}z^{n-1}+\overline{p_{0}}z^{n},\label{self}
\end{align}
 where the bar means complex conjugation. Notice that the conjugate,
reciprocal and inversive polynomials can also be defined without making
reference to the coefficients of $p(z)$: 
\begin{equation}
\overline{p}(z)=\overline{p(\overline{z})},\qquad p^{\ast}(z)=z^{n}p\left(1/z\right),\qquad p^{\dagger}(z)=z^{n}\overline{p\left(1/\overline{z}\right)}.\label{self2}
\end{equation}
\end{defn}
From these relations we plainly see that, if $\zeta_{1},\ldots,\zeta_{n}$
are the zeros of a complex polynomial $p(z)$ of degree $n$, then,
the zeros of $\overline{p}(z)$ are $\overline{\zeta_{1}},\ldots,\overline{\zeta_{n}}$,
the zeros of $p^{\ast}(z)$ are $1/\zeta_{1},\ldots,1/\zeta_{n}$
and, finally, the zeros of $p^{\dagger}(z)$ are $1/\overline{\zeta_{1}},\ldots,1/\overline{\zeta_{n}}$.
Thus, if $p(z)$ has $k$ zeros on $\mathbb{R}$, $l$ zeros on the
upper half-plane $\mathbb{C}^{+}=\left\{ z\in\mathbb{C}:\mathrm{Im}(z)>0\right\} $
and $m$ zeros in the lower half-plane $\mathbb{C}^{-}=\left\{ z\in\mathbb{C}:\mathrm{Im}(z)<0\right\} $
so that $k+l+m=n$, then $\overline{p}(z)$ will have the same number
$k$ of zeros on $\mathbb{R}$, $l$ zeros in $\mathbb{C}^{-}$ and
$m$ zeros in $\mathbb{C}^{+}$. Similarly, if $p(z)$ has $k$ zeros
on $\mathbb{S}$, $l$ zeros inside $\mathbb{S}$ and $m$ zeros outside
$\mathbb{S}$, so that $k+l+m=n$, then both $p^{\ast}(z)$ as $p^{\dagger}(z)$
will have the same number $k$ of zeros on $\mathbb{S}$, $l$ zeros
outside $\mathbb{S}$ and $m$ zeros inside $\mathbb{S}$.

These properties encourage us to introduce the following classes of
polynomials:
\begin{defn}
\label{selfies} A complex polynomial $p(z)$ is called\footnote{The reader should be aware that there is no standard in naming these
polynomials. For instance, what we call here self-inversive polynomials
are sometimes called self-reciprocal polynomials. What we mean positive
self-reciprocal polynomials are usually just called self-reciprocal
or yet palindrome polynomials (because their coefficients are the
same whether they are read from forwards or backwards), as well as,
negative self-reciprocal polynomials are usually called skew-reciprocal,
anti-reciprocal or yet anti-palindrome polynomials.} \emph{self-conjugate} (SC), \emph{self-reciprocal} (SR) or \emph{self-inversive}
(SI) if, for any zero $\zeta$ of $p(z)$, the complex-conjugate $\overline{\zeta}$,
the reciprocal $1/\zeta$ or the reciprocal of the complex-conjugate
$1/\overline{\zeta}$ is also a zero of $p(z)$, respectively.
\end{defn}
Thus, the zeros of any SC polynomial are all symmetric to the real
line $\mathbb{R}$, while the zeros of the any SI polynomial are symmetric
to the unit circle $\mathbb{S}$. The zeros of any SR polynomial are
obtained by an inversion with respect to the unit circle followed
by a reflection in the real line. From this we can establish the following:
\begin{thm}
If $p(z)$ is an SC polynomial of odd degree, then it necessarily
has at least one zero on $\mathbb{R}$. Similarly, if $p(z)$ is an
SR or SI polynomial of odd degree, then it necessarily has at least
one zero on $\mathbb{S}$.
\end{thm}
\begin{proof}
From Definition \ref{selfies} it follows that the number of non-real
zeros of an SC polynomial $p(z)$ can only occur in (conjugate) pairs;
thus, if $p(z)$ has odd degree, then at least one zero of it must
be real. Similarly, the zeros of $p^{\dagger}(z)$ or $p^{\ast}(z)$
that have modulus different from $1$ can only occur in (inversive
or reciprocal) pairs as well; thus, if $p(z)$ has odd degree then
at least one zero of it must lie on $\mathbb{S}$.
\end{proof}
\begin{thm}
The necessary and sufficient condition for a complex polynomial $p(z)$
to be SC, SR or SI is that there exists a complex number $\omega$
of modulus $1$ so that one of the following relations respectively
holds:
\begin{equation}
p(z)=\omega\overline{p}(z),\qquad p(z)=\omega p^{\ast}(z),\qquad p(z)=\omega p^{\dagger}(z).\label{self3}
\end{equation}
\end{thm}
\begin{proof}
It is clear in view of (\ref{self}) and (\ref{self2}) that these
conditions are sufficient. We need to show, therefore, that these
conditions are also necessary. Let us suppose first that $p(z)$ is
SC. Then, for any zero $\zeta$ of $p(z)$ the complex-conjugate number
$\overline{\zeta}$ is also a zero of it. Thus we can write, 
\begin{equation}
p(z)=p_{n}\prod_{k=1}^{n}\left(z-\overline{\zeta_{k}}\right)=p_{n}\prod_{k=1}^{n}\overline{\left(\overline{z}-\zeta_{k}\right)}=\left(p_{n}/\overline{p_{n}}\right)\overline{p(\overline{z})}=\omega\bar{p}(z),
\end{equation}
with $\omega=p_{n}/\overline{p_{n}}$ so that $|\omega|=\left|p_{n}/\overline{p_{n}}\right|=1$.
Now, let us suppose that $p(z)$ is SR. Then, for any zero $\zeta$
of $p(z)$ the reciprocal number $1/\zeta$ is also a zero of it;
thus, 
\begin{equation}
p(z)=p_{n}\prod_{k=1}^{n}\left(z-\frac{1}{\zeta_{k}}\right)=\frac{\left(-1\right)^{n}z^{n}p_{n}}{\zeta_{1}\cdots\zeta_{n}}\prod_{k=1}^{n}\left(\frac{1}{z}-\zeta_{k}\right)=\frac{\left(-1\right)^{n}z^{n}}{\zeta_{1}\cdots\zeta_{n}}p\left(\frac{1}{z}\right)=\omega p^{\ast}(z),
\end{equation}
with $\omega=\left(-1\right)^{n}/\left(\zeta_{1}\ldots\zeta_{n}\right)=p_{n}/p_{0}$;
now, for any zero $\zeta$ of $p(z)$ (which is necessarily different
from zero if $p(z)$ is SR), there will be another zero whose value
is $1/\zeta$ so that $\left|\zeta_{1}\ldots\zeta_{n}\right|=1$,
which implies $|\omega|=1$. The proof for SI polynomials is analogous
and will be concealed; it follows that $\omega=p_{n}/\overline{p_{0}}$
in this case. 
\end{proof}
Now, from (\ref{self}), (\ref{self2}) and (\ref{self3}) we can
conclude that the coefficients of an SC, an SR and an SI polynomial
of degree $n$ satisfy, respectively, the following relations:
\begin{equation}
p_{k}=\omega\overline{p_{k}},\qquad p_{k}=\omega p_{n-k},\qquad p_{k}=\omega\overline{p_{n-k}},\qquad|\omega|=1,\qquad0\leqslant k\leqslant n.
\end{equation}

We highlight that any real polynomial is SC --- in fact, many theorems
which are valid for real polynomials are also valid for, or can be
easily extended to, SC polynomials. 

There also exist polynomials whose zeros are symmetric with respect
to \emph{both} the real line $\mathbb{R}$ and the unit circle $\mathbb{S}$.
A polynomial $p(z)$ with this double symmetry is, at the same time,
SC and SI (and, hence, SR as well). This is only possible if all the
coefficients of $p(z)$ are real, which implies that $\omega=\pm1$.
This suggests the following additional definitions: 
\begin{defn}
A real self-reciprocal polynomial $p(z)$ that satisfies the relation
$p(z)=\omega z^{n}p(1/z)$ will be called a\emph{ positive self-reciprocal
}(PSR)\emph{ }polynomial if $\omega=1$ and a\emph{ negative self-reciprocal}
(NSR) polynomial\emph{ }if $\omega=-1$.
\end{defn}
Thus, the coefficients of any PSR polynomial $p(z)=p_{0}+\cdots+p_{n}z^{n}$
of degree $n$ satisfy the relations $p_{k}=p_{n-k}$ for $0\leqslant k\leqslant n$,
while the coefficients of any NSR polynomial $p(z)$ of degree $n$
satisfy the relations $p_{k}=-p_{n-k}$ for $0\leqslant k\leqslant n$;
this last condition implies that the middle coefficient of an NSR
polynomial of even degree is always zero. 

Some elementary properties of PSR and NSR polynomials are the following:
first, notice that, if $\zeta$ is a zero of any PSR or NSR polynomial
$p(z)$ of degree $n\geqslant4$, then the three complex numbers $1/\zeta$,
$\overline{\zeta}$ and $1/\overline{\zeta}$ are also zeros of $p(z)$.
In particular, the number of zeros of such polynomials which are neither
on $\mathbb{S}$ or $\mathbb{R}$ is always a multiple of $4$. Besides,
any NSR polynomial has $z=1$ as a zero and $p(z)/(z-1)$ is PSR;
further, if $p(z)$ has even degree then $z=-1$ is also a zero of
it and $p(z)/(z^{2}-1)$ is a PSR polynomial of even degree. In a
similar way, any PSR polynomial $p(z)$ of odd degree has $z=-1$
as a zero and $p(z)/(z+1)$ is also PSR. The product of two PSR, or
two NSR, polynomials is PSR, while the product of a PSR polynomial
with an NSR polynomial is NSR. These statements follow directly from
the definitions of such polynomials.

We also mention that any PSR polynomial of even degree (say, $n=2m$)
can be written in the following form: 
\begin{equation}
p(z)=z^{m}\left[p_{0}\left(z^{m}+\frac{1}{z^{m}}\right)+p_{1}\left(z^{m-1}+\frac{1}{z^{m-1}}\right)+\cdots+p_{m-1}\left(z+\frac{1}{z}\right)\right]+p_{m},
\end{equation}
an expression that is obtained by using the relations $p_{k}=p_{2m-k}$,
$0\leqslant k\leqslant2m$ and gathering the terms of $p(z)$ with
the same coefficients. Furthermore, the expression $Z_{s}(z)=\left(z^{s}+z^{-s}\right)$
for any integer $s$ can be written as a polynomial of degree $s$
in the new variable $x=z+1/z$ (the proof follows easily by induction
over $s$); thus, we can write $p(z)=z^{m}q\left(x\right)$, where
$q(x)=q_{0}+\cdots+q_{m}x^{m}$ is such that the coefficients $q_{0},\ldots,q_{m}$
are certain functions of $p_{0},\ldots,p_{m}$. From this we can state
the following: 
\begin{thm}
Let $p(z)$ be a PSR polynomial of even degree $n=2m$. For each pair
$\zeta$ and $1/\zeta$ of self-reciprocal zeros of $p(z)$ that lie
on $\mathbb{S}$, there is a corresponding zero $\xi$ of the polynomial
$q(x)$, as defined above, in the interval $\left[-2,2\right]$ of
the real line.
\end{thm}
\begin{proof}
For each zero $\zeta$ of $p(x)$ that lie on $S$, write $\zeta=\mathrm{e}^{i\theta}$
for some $\theta\in\mathbb{R}$. Besides, as $q(x)=q(z+1/z)=p(z)/z^{m}$,
it follows that $\xi=\zeta+1/\zeta=2\cos\theta$ will be a zero of
$q(x)$. This shows us that $\xi$ is limited to the interval $\left[-2,2\right]$
of the real line. Finally, notice that the reciprocal zero $1/\zeta$
of $p(z)$ is mapped to the same zero $\xi$ of $q(x)$.
\end{proof}
Finally, remembering that the Chebyshev polynomials of first kind,
$T_{n}(z)$, are defined by the formula $T_{n}\left[\tfrac{1}{2}\left(z+z^{-1}\right)\right]=\tfrac{1}{2}\left(z^{n}+z^{-n}\right)$
for $z\in\mathbb{C}$, it follows as well that $q(x)$, and hence
any PSR polynomial, can be written as a linear combination of Chebyshev
polynomials: 
\begin{equation}
q(x)=2\left[p_{0}T_{m}(x)+p_{1}T_{m-1}(x)+\cdots+p_{m-1}T_{1}(x)+\tfrac{1}{2}p_{m}T_{0}(x)\right].
\end{equation}

\section{How these polynomials are related to each other?\label{SectionRelationship} }

In this section, we shall analyze how SC, SR and SI polynomials are
related to each other. Let us begin with the relationship between
the SR and SI polynomials, which is actually very simple: indeed,
from (\ref{self}), (\ref{self2}) and (\ref{self3}) we can see that
each one is nothing but the conjugate polynomial of the other, that
is,
\begin{equation}
p^{\dagger}(z)=\overline{p^{\ast}}(z)=\overline{p^{\ast}(\overline{z})},\qquad\text{and}\qquad p^{\ast}(z)=\overline{p^{\dagger}}(z)=\overline{p^{\dagger}(\overline{z})}.
\end{equation}
Thus, if $p(z)$ is an SR (SI) polynomial, then $\overline{p}(z)$
will be SI (SR) polynomial. Because of this simple relationship, several
theorems which are valid for SI polynomials are also valid for SR
polynomials and vice versa.

The relationship between SC and SI polynomials is not so easy to perceive.
A way of revealing their connection is to make use of a suitable pair
of \emph{Möbius transformations}, that maps the unit circle onto the
real line and vice versa, which is often called \emph{Cayley transformations}
, defined through the formulas: 
\begin{equation}
M(z)=\left(z-i\right)/\left(z+i\right),\qquad\text{and}\qquad W(z)=-i\left(z+1\right)/\left(z-1\right).
\end{equation}
This approach was developed in \cite{Vieira2018}, where some algorithms
for counting the number of zeros that a complex polynomial has on
the unit circle were also formulated. 

It is a easy matter to verify that $M(z)$ maps $\mathbb{R}$ onto
$\mathbb{S}$ while $W(z)$ maps $\mathbb{S}$ onto $\mathbb{R}$.
Besides, $M(z)$ maps the upper (lower) half-plane to the interior
(exterior) of $\mathbb{S}$, while $W(z)$ maps the interior (exterior)
of $\mathbb{S}$ onto the upper (lower) half-plane. Notice that $W(z)$
can be thought as the inverse of $M(z)$ in the Riemann sphere $\mathbb{C}_{\infty}=\mathbb{C}\cup\{\infty\}$,
if we further assume that $M(-i)=\infty$, $M(\infty)=1$, $W(1)=\infty$
and $W(\infty)=-i$. 

Given a polynomial $p(z)$ of degree $n$, we define two \emph{Möbius-transformed
polynomials}, namely, 
\begin{equation}
Q(z)=(z+i)^{n}p(M(z)),\qquad\text{and}\qquad T(z)=(z-1)^{n}p(W(z)).
\end{equation}

The following theorem shows us how the zeros of $Q(z)$ and $T(z)$
are related with the zeros of $p(z)$:
\begin{thm}
\label{TheoremZeros}Let $\zeta_{1},\ldots,\zeta_{n}$ denote the
zeros of $p(z)$ and $\eta_{1},\ldots,\eta_{n}$ the respective zeros
of $Q(z)$. Provided $p(1)\neq0$, we have that $\eta_{1}=W(\zeta_{1}),\ldots,\eta_{n}=W(\zeta_{n})$.
Similarly, if $\tau_{1},\ldots\tau_{n}$ are the zeros of $T(z)$,
then we have $\tau_{1}=M(\zeta_{1}),\ldots,\tau_{n}=M(\zeta_{n})$,
provided that $p(-i)\neq0$. 
\end{thm}
\begin{proof}
In fact, inverting the expression for $Q(z)$ and evaluating it in
any zero $\zeta_{k}$ of $p(z)$ we get that $p(\zeta_{k})=(-i/2)^{n}(\zeta_{k}-1)^{n}Q(W(\zeta_{k}))=0$
for $0\leqslant k\leqslant n$. Provided that $z=1$ is not a zero
of $p(z)$ we get that $\eta_{k}=W(\zeta_{k})$ is a zero of $Q(z)$.
The proof for the zeros of $T(z)$ is analogous.
\end{proof}
This result also shows that $Q(z)$ and $T(z)$ have the same degree
as $p(z)$ whenever $p(1)\neq0$ or $p(-i)\neq0$, respectively. In
fact, if $p(z)$ has a zero at $z=1$ of multiplicity $m$ then $Q(z)$
will be a polynomial of degree $n-m$, the same being true for $T(z)$
if $p(z)$ has a zero of multiplicity $m$ at $z=-i$. This can be
explained by the fact that the points $z=1$ and $z=-i$ are mapped
to infinity by $W(z)$ and $M(z)$, respectively.$ $

The following theorem shows that the set of SI polynomials are isomorphic
to the set of SC polynomials:
\begin{thm}
\label{Theorem-SI-SC} Let $p(z)$ be an SI polynomial. Then, the
transformed polynomial $Q(z)=(z+i)^{n}p(M(z))$ is an SC polynomial.
Similarly, if $p(z)$ is an SC polynomial, then $T(z)=(z-1)^{n}p(W(z))$
will be an SI polynomial. 
\end{thm}
\begin{proof}
Let $\zeta$ and $1/\bar{\zeta}$ be two inversive zeros an SI polynomial
$p(z)$. Then, according to Theorem \ref{TheoremZeros}, the corresponding
zeros of $Q(z)$ will be: 
\begin{equation}
W(\zeta)=-i\frac{\zeta+1}{\zeta-1}=\eta\qquad\text{and}\qquad W\left(\frac{1}{\overline{\zeta}}\right)=-i\frac{1/\overline{\zeta}+1}{1/\overline{\zeta}-1}=i\frac{\overline{\zeta}+1}{\overline{\zeta}-1}=\overline{W(\zeta)}=\overline{\eta}.
\end{equation}
Thus, any pair of zeros of $p(z)$ that are symmetric to the unit
circle are mapped in zeros of $Q(z)$ that are symmetric to the real
line; because $p(z)$ is SI, it follows that $Q(z)$ is SC. Conversely,
let $\zeta$ and $\overline{\zeta}$ be two zeros of an SC polynomial
$p(z)$; then the corresponding zeros of $T(z)$ will be:
\begin{equation}
M(\zeta)=\frac{\zeta-i}{\zeta+i}=\tau\qquad\text{and}\qquad M\left(\overline{\zeta}\right)=\frac{\overline{\zeta}-i}{\overline{\zeta}+i}=-\frac{1/\overline{\zeta}+i}{1/\overline{\zeta}-i}=\frac{1}{\overline{M\left(\zeta\right)}}=\frac{1}{\overline{\tau}}.
\end{equation}
Thus, any pair of zeros of $p(z)$ that are symmetric to the real
line are mapped in zeros of $T(z)$ that are symmetric to the unit
circle. Because $p(z)$ is SC, it follows that $T(z)$ is SI.
\end{proof}
We can also verify that any SI polynomial with $\omega=1$ is mapped
to a real polynomial through $M(z)$, as well as, any real polynomial
is mapped to an SI polynomial with $\omega=1$ through $W(z)$. Thus,
the set of SI polynomials with $\omega=1$ is isomorphic to the set
of real polynomials. Besides, an SI polynomial with $\omega\neq1$
can be transformed into another one with $\omega=1$ by performing
a suitable uniform rotation of its zeros. It can also be shown that
the action of the Möbius transformation over a PSR polynomial leads
to a real polynomial that has only even powers. See \cite{Vieira2018}
for more. 

\section{Zeros location theorems \label{SectionTheorems}}

In this section, we shall discuss some theorems regarding the distribution
of the zeros of SC, SR and SI polynomials on the complex plane. Some
general theorems relying on the number of zeros that an arbitrary
complex polynomial has inside, on, or outside $\mathbb{S}$ are also
discussed. To save space, we shall not present the proofs of these
theorems, which can be found in the original works. Other related
theorems can be found in Marden's book \cite{Marden1966}.

\subsection{Polynomials that do not necessarily have symmetric zeros}

The following theorems are classics (see \cite{Marden1966} for the
proofs):
\begin{thm}
\textbf{\emph{(Rouché)}}. Let $q(z)$ and $r(z)$ be polynomials such
that $|q(z)|<|r(z)|$ along all points of $\mathbb{S}$. Then, the
polynomial $p(z)=q(z)+r(z)$ has the same number of zeros inside $\mathbb{S}$
as the polynomial $r(z)$, counted with multiplicity.
\end{thm}
Thus, if a complex polynomial $p(z)=p_{0}+\cdots+p_{k}z^{k}+\cdots+p_{n}z^{n}$
of degree $n$ is such that $\left|p_{k}\right|>\left|p_{0}+\cdots+p_{k-1}+p_{k+1}+\cdots+p_{n}\right|$,
then $p(z)$ will have exactly $k$ zeros inside $\mathbb{S}$, counted
with multiplicity.
\begin{thm}
\textbf{\emph{(Gauss \& Lucas)}} The zeros of the derivative $p^{\prime}(z)$
of a polynomial $p(z)$ lie all within the convex hull of the zeros
of the $p(z)$.
\end{thm}
Thereby, if a polynomial $p(z)$ has all its zeros on $\mathbb{S}$,
then all the zeros of $p^{\prime}(z)$ will lie in or on $\mathbb{S}$.
In particular, the zeros of $p^{\prime}(z)$ will lie on $\mathbb{S}$
if, and only if, they are multiple zeros of $p(z)$. 
\begin{thm}
\textbf{\emph{(Cohn)}} \label{TheoremCohn1}A necessary and sufficient
condition for all the zeros of a complex polynomial $p(z)$ to lie
on $\mathbb{S}$ is that $p(z)$ is SI and that its derivative \textup{$p^{\prime}(z)$}
does not have any zero outside\textup{ $\mathbb{S}$. }
\end{thm}
Cohn introduced his theorem in \cite{Cohn1922}. Bonsall \& Marden
presented a simpler proof of Conh's theorem  in \cite{Bonsall1952}
(see also \cite{Ancochea1953}) and applied it to SI polynomials ---
in fact, this was probably the first paper to use the expression ``self-inversive''.
Other important result of Cohn is the following: all the zeros of
a complex polynomial $p(z)=p_{n}z^{n}+\cdots+p_{0}$ will lie on $\mathbb{S}$
if, and only if, $|p_{n}|=|p_{0}|$ and all the zeros of $p(z)$ do
not lie outside $\mathbb{S}$. 

Restricting ourselves to polynomials with real coefficients, Eneström
\& Kakeya \cite{Enestrom1893,Enestrom1920,Kakeya1912} independently
presented the following theorem:
\begin{thm}
\textbf{\emph{(Eneström \& Kakeya) \label{TheoremEK}}} Let $p(z)$
be a polynomial of degree $n$ with real coefficients. If its coefficients
are such that $0<p_{0}\leqslant p_{1}\leqslant\cdots\leqslant p_{n-1}\leqslant p_{n}$,
then all the zeros of $p(z)$ lie in or on $\mathbb{S}$. Likewise,
if the coefficients of $p(z)$ are such that $0<p_{n}\leqslant p_{n-1}\leqslant\cdots\leqslant p_{1}\leqslant p_{0}$,
then all the zeros of $p(z)$ lie on or out $\mathbb{S}$.
\end{thm}
The following theorems are relatively more recent. The distribution
of the zeros of a complex polynomial regarding the unit circle $\mathbb{S}$
was presented by Marden in \cite{Marden1966} and slightly enhanced
by Jury in \cite{Jury1964}:
\begin{thm}
\textbf{\emph{(Marden \& Jury)}} \label{TheoremMarden}Let $p(z)$
be a complex polynomial of degree $n$ and $p^{\ast}(z)$ its reciprocal.
Construct the sequence of polynomials $P_{j}(z)=\sum_{k=0}^{n-j}P_{j,k}z^{k}$
such that $P_{0}(z)=p(z)$ and $ $$P_{j+1}(z)=\overline{p_{j,0}}P_{j}(z)-\overline{p_{j,n-j}}P_{j}^{\ast}(z)$
for $0\leqslant j\leqslant n-1$ so that we have the relations $p_{j+1,k}=\overline{p_{j,0}}p_{j,k}-p_{j,n-j}\overline{p_{j,n-j-k}}$.
Let $\delta_{j}$ denote the constant terms of the polynomials $P_{j}(z)$,
i.e., $\delta_{j}=p_{j,0}$ and $\varDelta_{k}=\delta_{1}\cdots\delta_{k}$.
Thus, if $N$ of the products $\varDelta_{k}$ are negative and $n-N$
of the products $\varDelta_{k}$ are positive so that none of them
are zero, then $p(z)$ has $N$ zeros inside $\mathbb{S}$, $n-N$
zeros outside $\mathbb{S}$ and no zero on $\mathbb{S}$. On the other
hand, if $\varDelta_{k}\neq0$ for some $k<n$ but $P_{k+1}(z)=0$,
then $p(z)$ has either $n-k$ zeros on $\mathbb{S}$ or $n-k$ zeros
symmetric to $\mathbb{S}$. It has additionally $N$ zeros inside
$\mathbb{S}$ and $k-N$ zeros outside $\mathbb{S}$.
\end{thm}
A simple necessary and sufficient condition for all the zeros of a
complex polynomial to lie on $\mathbb{S}$ was presented by Chen in
\cite{Chen1995}: 
\begin{thm}
\textbf{\emph{(Chen)\label{TheoremChen}}} A necessary and sufficient
condition for all the zeros of a complex polynomial $p(z)$ of degree
$n$ to lie on $\mathbb{S}$ is that there exists a polynomial $q(z)$
of degree $n-m$ whose zeros are all in or on $\mathbb{S}$ and such
that $p(z)=z^{m}q(z)+\omega q^{\dagger}(z)$ for some complex number
$\omega$ of modulus $1$.
\end{thm}
We close this section by mentioning that there exist many other well-known
theorems regarding the distribution of the zeros of complex polynomials.
We can cite, for example, the famous \emph{rule of Descartes} (the
number of positive zeros of a real polynomial is limited from above
by the number of sign variations in the ordered sequence of its coefficients),
the \emph{Sturm Theorem} (the exact number of zeros that a real polynomial
has in a given interval $(a,b]$ of the real line is determined by
the formula $N=\mathrm{var}\left[S(b)\right]-\mathrm{var}\left[S(a)\right]$,
where $\mathrm{var}\left[S(\xi)\right]$ means the number of sign
variations of the Sturm sequence $S(x)$ evaluated at $x=\xi$) and
\emph{Kronecker Theorem} (if all the zeros of a monic polynomial with
integer coefficients lie on the unit circle, then all these zeros
are indeed roots of unity), see \cite{Marden1966} for more. There
are still other important theorems relying on matrix methods and quadratic
forms that were developed by several authors as Cohn, Schur, Hermite,
Sylvester, Hurwitz, Krein, among others, see \cite{KreinNaimark1981}.

\subsection{Real self-reciprocal polynomials}

Let us now consider real SR polynomials. The theorems below are usually
applied to PSR polynomials, but some of them can be extended to NSR
polynomials as well. 

An analogue of Eneström-Kakeya theorem for PSR polynomials was found
by Chen in \cite{Chen1995} and then, in a slightly stronger version,
by Chinen in \cite{Chinen2008}:
\begin{thm}
\textbf{\emph{(Chen \& Chinen)}} Let $p(z)$ be a PSR polynomial of
degree $n$ that is written in the form $p(z)=p_{0}+p_{1}z+\cdots+p_{k}z^{k}+p_{k}z^{n-k}+p_{k-1}z^{n-k+1}+\cdots+p_{0}z^{n}$
and such that $0<p_{k}<p_{k-1}<\cdots<p_{1}<p_{0}$. Then all the
zeros of $p(z)$ are on $\mathbb{S}$.
\end{thm}
Going in the same direction, Choo found in \cite{Choo2011} the following
condition: 
\begin{thm}
\textbf{\emph{(Choo)}} Let $p(z)$ be a PSR polynomial of degree $n$
and such that its coefficients satisfy the following conditions: $np_{n}\geqslant\left(n-1\right)p_{n-1}\geqslant\cdots\geqslant\left(k+1\right)p_{k+1}>0$
and $\left(k+1\right)p_{k+1}\geqslant\sum_{j=0}^{k}\left|\left(j+1\right)p_{j+1}-jp_{j}\right|$
for $0\leqslant k\leqslant n-1$ Then, all the zeros of $p(z)$ are
on $\mathbb{S}$. 
\end{thm}
Lakatos discussed the separation of the zeros on the unit circle of
PSR polynomials in \cite{Lakatos2002}; she also found several sufficient
conditions for their zeros to be all on $\mathbb{S}$. One of the
main theorems is the following:
\begin{thm}
\textbf{\emph{(Lakatos)}} \label{TheoremLakatos1}Let $p(z)$ be a
PSR polynomial of degree $n>2$. If $\left|p_{n}\right|\geqslant\sum_{k=1}^{n-1}\left|p_{n}-p_{k}\right|$,
then all the zeros of $p(z)$ lie on $\mathbb{S}$. Moreover, the
zeros of $p(z)$ are all simple, except when the equality takes place.
\end{thm}
For PSR polynomials of odd degree, Lakatos \& Losonczi \cite{LakatosLosonczi2003}
found a stronger version of this result:
\begin{thm}
\textbf{\emph{(Lakatos \& Losonczi)}} \label{TheoremLakatosLosonczi1}Let
$p(z)$ be a PSR polynomial of odd degree, say $n=2m+1$. If $\left|p_{2m+1}\right|\geqslant\cos^{2}(\phi_{m})\sum_{k=1}^{2m}\left|p_{2m+1}-p_{k}\right|$,
where $\phi_{m}=\pi/\left[4\left(m+1\right)\right]$, then all the
zeros of $p(z)$ lie on $\mathbb{S}$. The zeros are simple except
when the equality is strict.
\end{thm}
Theorem \ref{TheoremLakatos1} was generalized further by Lakatos
\& Losonczi in \cite{LakatosLosonczi2007}:
\begin{thm}
\textbf{\emph{(Lakatos \& Losonczi)}} All zeros of a PSR polynomial
$p(z)$ of degree $n>2$ lie on $\mathbb{S}$ if the following conditions
hold: $\left|p_{n}+r\right|\geqslant\sum_{k=1}^{n-1}\left|p_{k}-p_{n}+r\right|$,
$p_{n}r\geqslant0$ and $\left|p_{n}\right|\geqslant\left|r\right|$,
for $r\in\mathbb{R}$.
\end{thm}
Other conditions for all the zeros of a PSR polynomial to lie on $\mathbb{S}$
was presented by Kwon in \cite{Kwon2011}. In its simplest form, Kown's
theorem can be enunciated as follows:
\begin{thm}
\textbf{\emph{(Kwon)}} Let $p(z)$ be a PSR polynomial of even degree
$n\geqslant2$ whose leading coefficient $p_{n}$ is positive and
$p_{0}\leqslant p_{1}\leqslant\cdots\leqslant p_{n}$. In this case,
all the zeros of $p(z)$ will lie on $\mathbb{S}$ if, either $p_{n/2}\geqslant\sum_{k=0}^{n}\left|p_{k}-p_{n/2}\right|$,
or $p(1)\geqslant0$ and $p_{n}\geqslant\frac{1}{2}\sum_{k=1}^{n-1}\left|p_{k}-p_{n/2}\right|$.
\end{thm}
Modified forms of this theorem hold for the PSR polynomials of odd
degree and for the case where the coefficients of $p(z)$ do not have
the ordination above --- see \cite{Kwon2011} for these cases. Kwon
also found conditions for all but two zeros of $p(z)$ to lie on $\mathbb{S}$
in \cite{Kwon2011B}, which is relevant to the theory of Salem polynomials
--- see Section \ref{SectionApplications}. 

Other interesting results are the following: Konvalina \& Matache
\cite{KonvalinaMatache2004} found conditions under which a PSR polynomial
has at least one non-real zero on $\mathbb{S}$. Kim \& Park \cite{KimPark2008}
and then Kim and Lee \cite{KimLee2010} presented conditions for which
all the zeros of certain PSR polynomials lie on $\mathbb{S}$ (some
open cases were also addressed by Botta \& al. in \cite{BottaBraccialiPereira2016}).
Suzuki \cite{Suzuki2012} presented necessary and sufficient conditions,
relying on matrix algebra and differential equations, for all the
zeros of PSR polynomials to lie on $\mathbb{S}$. In \cite{BottaMarquesMeneguette2014}
Botta \& al. studied the distribution of the zeros of PSR polynomials
with a small perturbation in their coefficients. Real SR polynomials
of height $1$ --- namely, special cases of \emph{Littlewood}, \emph{Newman}
and \emph{Borwein polynomials} --- were studied by several authors,
see \cite{ConreyEtal2000,Erdelyi2001B,Mossinghoff2003,Mercer2006,Mukunda2006,Drungilas2008,BaradaranTaghavi2014,BorweinEtal2015,DrungilasEtal2018}
and references therein\footnote{The zeros of such polynomials present a fractal behaviour, as was
first discovered by Odlyzko and Poonen in \cite{OdlyzkoPoonen1993}.}. Zeros of the so-called \emph{Ramanujan Polynomials} and generalizations
were analyzed in \cite{MurtySmythWang2011,LalinEtal2013,DiamantisEtal2018}.
Finally, the Galois theory of PSR polynomials was studied in \cite{Lindstrom2015}
by Lindstr{\o}m, who showed that any PSR polynomial of degree less
than $10$ can be solved by radicals.

\subsection{Complex self-reciprocal and self-inversive polynomials}

Let us consider now the case of complex SR polynomials and SI polynomials.
Here we remark that many of the theorems that hold for SI polynomials
either also hold for SR polynomials or can be easily adapted to this
case (the opposite is also true). 
\begin{thm}
\textbf{\emph{(Cohn)}} An SI polynomial $p(z)$ has as many zeros
outside $\mathbb{S}$ as does its derivative \textup{$p^{\prime}(z)$}. 
\end{thm}
This follows directly from Cohn's Theorem \ref{TheoremCohn1} for
the case where $p(z)$ is SI. Besides, we can also conclude from this
that the derivative of $p(z)$ has no zeros on $\mathbb{S}$ except
at the multiple zeros of $p(z)$. Furthermore, if an SI polynomial
$p(z)$ of degree $n$ has exactly $k$ zeros on $\mathbb{S}$, while
its derivative has exactly $l$ zeros in or on $\mathbb{S}$, both
counted with multiplicity, then $n=2(l+1)-k$. 

O'Hara \& Rodriguez \cite{OharaRodriguez1974} showed that the following
conditions are always satisfied by SI polynomials whose zeros are
all on $\mathbb{S}$:
\begin{thm}
\textbf{\emph{(O'Hara \& Rodriguez)}} Let $p(z)$ be an SI polynomial
of degree $n$ whose zeros are all on $\mathbb{S}$. Then, the following
inequality holds: $\sum_{j=0}^{n}|p_{j}|^{2}\leqslant\left\Vert p(z)\right\Vert ^{2}$,
where $\left\Vert p(z)\right\Vert $ denotes the maximum modulus of
$p(z)$ on the unit circle; besides, if this inequality is strict
then the zeros of $p(z)$ are rotations of $n$th roots of unity.
Moreover, the following inequalities are also satisfied: \textup{$\left|a_{k}\right|\leqslant\tfrac{1}{2}\left\Vert p(z)\right\Vert $
if $k\neq n/2$ and $\left|a_{k}\right|\leqslant\tfrac{\sqrt{2}}{2}\left\Vert p(z)\right\Vert $
for $k=n/2$.}
\end{thm}
Schinzel in \cite{Schinzel2005}, generalized Lakatos Theorem \ref{TheoremLakatos1}
for SI polynomials:
\begin{thm}
\textbf{\emph{(Schinzel)}} Let $p(z)$ be an SI polynomial of degree
$n$. If the inequality $\left|p_{n}\right|\geqslant\inf_{a,b\in\mathbb{C}:|b|=1}\sum_{k=0}^{n}\left|ap_{k}-b^{n-k}p_{n}\right|$,
then all the zeros of $p(z)$ lie on $\mathbb{S}$. These zero are
simple whenever the equality is strict.
\end{thm}
In a similar way, Losonczi \& Schinzel \cite{LosoncziSchinzel2007}
generalized theorem \ref{TheoremLakatosLosonczi1} for the SI case: 
\begin{thm}
\textbf{\emph{(Losonczi \& Schinzel)}} Let $p(z)$ be an SI polynomial
of odd degree, i.e. $n=2m+1$. If $\left|p_{2m+1}\right|\geqslant\cos^{2}\left(\phi_{m}\right)\inf_{a,b\in\mathbb{C}:|b|=1}\sum_{k=1}^{2m+1}\left|ap_{k}-b^{2m+1-k}p_{2m+1}\right|$,
where $\phi_{m}=\pi/\left[4\left(m+1\right)\right]$, then all the
zeros of $p(z)$ lie on $\mathbb{S}$. The zeros are simple except
when the equality is strict.
\end{thm}
Another sufficient condition for all the zeros of an SI polynomial
to lie on $\mathbb{S}$ was presented by Lakatos \& Losonczi in \cite{LakatosLosonczi2004}:
\begin{thm}
\textbf{\emph{(Lakatos \& Losonczi)}} Let $p(z)$ be an SI polynomial
of degree $n$ and suppose that the inequality $\left|p_{n}\right|\geqslant\frac{1}{2}\sum_{k=1}^{n-1}\left|p_{k}\right|$
holds. Then, all the zeros of $p(z)$ lie on $\mathbb{S}$. Moreover,
the zeros are all simple except when an equality takes place.
\end{thm}
In \cite{LakatosLosonczi2009} Lakatos \& Losonczi also formulated
a theorem that contains as special cases many of the previous results:
\begin{thm}
\textbf{\emph{(Lakatos \& Losonczi)}} Let $p(z)=p_{0}+\cdots+p_{n}z^{n}$
be an SI polynomial of degree $n\geqslant2$ and $a$, $b$ and $c$
be complex numbers such that $a\neq0$, $|b|=1$ and $c/p_{n}\in\mathbb{R}$,
$0\leqslant c/p_{n}\leqslant1$. If $\left|p_{n}+c\right|\geqslant\left|ap_{0}-b^{n}p_{n}\right|+\sum_{k=1}^{n-1}\left|ap_{k}-b^{n-k}\left(c-p_{n}\right)\right|+\left|ap_{n}-p_{n}\right|$,
then, all the zeros of $p(z)$ lie on $\mathbb{S}$. Moreover, these
zeros are simple if the inequality is strict.
\end{thm}
In \cite{Losonczi2006} Losonczi presented the following necessary
and sufficient conditions for all the zeros of a (complex) SR polynomial
of even degree to lie on $\mathbb{S}$: 
\begin{thm}
\textbf{\emph{(Losonczi)}} Let $p(z)$ be a monic complex SR polynomial
of even degree, say $n=2m$. Then, all the zeros of $p(z)$ will lie
on $\mathbb{S}$ if, and only if, there exist real numbers $\alpha_{1},\ldots,\alpha_{2m}$,
all with moduli less than or equal to $2$, that satisfy the inequalities:
$p_{k}=\left(-1\right)^{k}\sum_{l=0}^{\left[k/2\right]}\binom{m-k+2l}{l}\sigma_{k-2l}^{2m}\left(\alpha_{1},\ldots,\alpha_{2m}\right)$,
$0\leqslant k\leqslant m$, where $\sigma_{k}^{2m}\left(\alpha_{1},\ldots,\alpha_{2m}\right)$
denotes the $k$th elementary symmetric function in the $2m$ variables
$\alpha_{1},\ldots,\alpha_{2m}$. 
\end{thm}
Losonczi in \cite{Losonczi2006} also showed that if all the zeros
of a complex monic reciprocal polynomial are on $\mathbb{S}$ then
its coefficients are all real and satisfy the inequality $\left|p_{n}\right|\leqslant\binom{n}{k}$
for $0\leqslant k\leqslant n$.

The theorems above give conditions for \emph{all} the zeros of SI
or SR polynomials to lie on $\mathbb{S}$. In many cases, however,
we need to verify if a polynomial has a \emph{given number} of zeros
(or \emph{none}) on the unit circle. Considering this problem, Vieira
in \cite{Vieira2017} found sufficient conditions for an SI polynomial
of degree $n$ to have a determined number of zeros on the unit circle.
In terms of the length $L\left[p(z)\right]=\left|p_{0}\right|+\cdots+\left|p_{n}\right|$
of a polynomial $p(z)$ of degree $n$, this theorem can be stated
as follows: 
\begin{thm}
\label{TheoremVieira}\textbf{\emph{(Vieira)}} Let $p(z)$ be an SI
polynomial of degree $n$. If the inequality $\left|p_{n-m}\right|\geqslant\tfrac{1}{4}\left(\tfrac{n}{n-m}\right)L\left[p(z)\right]$,
$m<n/2$, holds true, then $p(z)$ will have exactly $n-2m$ zeros
on $\mathbb{S}$; besides, all these zeros are simple when the inequality
is strict. Moreover, $p(z)$ will have no zero on $\mathbb{S}$ if,
for $n$ even and $m=n/2$, the inequality $\left|p_{m}\right|>\frac{1}{2}L\left[p(z)\right]$
is satisfied.
\end{thm}
The case $m=0$ corresponds to Lakatos \& Losonczi Theorem \ref{TheoremLakatos1}
for all the zeros of $p(z)$ to lie on $\mathbb{S}$. The necessary
counterpart of this theorem was considered by Stankov in \cite{Stankov2018},
with an application to the theory of Salem numbers --- see Section
\ref{SectionSalem}.

Other results on the distribution of zeros of SI polynomials include
the following: Sinclair \& Vaaler \cite{SinclairVaaler2008} showed
that a monic SI polynomial $p(z)$ of degree $n$ satisfying the inequalities
$L^{r}\left[p(z)\right]\leqslant2+2^{r}\left(n-1\right)^{1-r}$ or
$L^{r}\left[p(z)\right]\leqslant2+2^{r}\left(l-2\right)^{1-r}$, where
$r\geqslant1$, $L^{r}\left[p(z)\right]=\left|p_{0}\right|^{r}+\cdots+\left|p_{n}\right|^{r}$
and $l$ is the number of non-null terms of $p(z)$, has all their
zeros on $\mathbb{S}$; the authors also studied the geometry of SI
polynomials whose zeros are all on $\mathbb{S}$. Choo \& Kim applied
Theorem \ref{TheoremChen} to SI polynomials in \cite{ChooKim2013}.
Hypergeometric polynomials with all their zeros on $\mathbb{S}$ were
considered in \cite{Area2013,Dimitrov2013}. Kim \cite{Kim2013} also
obtained SI polynomials which are related to Jacobi polynomials. Ito
\& Wimmer \cite{ItoWimmer2016} studied SI polynomial operators in
Hilbert space whose spectrum is on $\mathbb{S}$. 

\section{Where these polynomials are found?\label{SectionApplications} }

In this section, we shall briefly discuss some important or recent
applications of the theory of polynomials with symmetric zeros. We
remark, however, that our selection is by no means exhaustive: for
example, SR and SI polynomials also find applications in many fields
of mathematics (e.g., information and coding theory \cite{JoynerKim2011},
algebraic curves over a finite field and cryptography \cite{Joyner2013},
elliptic functions \cite{JoynerShaska2018}, number theory \cite{Mckee2008}
etc.) and physics (e.g., Lee-Yang theorem in statistical physics \cite{YangLee1952},
Poincaré Polynomials defined on Calabi-Yau manifolds of superstring
theory \cite{He2011} etc.). 

\subsection{Polynomials with small Mahler measure \label{SectionSalem}}

Given a monic polynomial $p(z)$ of degree $n$, with integer coefficients,
the Mahler measure of $p(z)$, denoted by $M[p(z)]$, is defined as
the product of the modulus of all those zeros of $p(z)$ that lie
in the exterior of $\mathbb{S}$ \cite{Everest2013}. That is, 
\begin{equation}
M[p(z)]=\prod_{i=1}^{n}\max\left\{ 1,\left|\zeta_{i}\right|\right\} ,\label{Mahler}
\end{equation}
where $\zeta_{1},\ldots,\zeta_{n}$ are the zeros\footnote{The Mahler measure of a monic integer polynomial $p(z)$ can also
be defined without making reference to its zeros through the formula
$M[p(z)]=\exp\left\{ \int_{0}^{1}\log\left[p\left(\mathrm{e}^{2\pi it}\right)\right]\mathrm{d}t\right\} $
--- see \cite{Everest2013}.} of $p(z)$. Thus, if a monic integer polynomial $p(z)$ has all its
zeros in or on the unit circle, we have $M[p(z)]=1$; in particular,
all cyclotomic polynomials (which are PSR polynomials whose zeros
are the primitive roots of unity, see \cite{Marden1966}) have Mahler
measure equal to $1$. In a sense, the Mahler measure of a polynomial
$p(z)$ measures how close it is to the cyclotomic polynomials. Therefore,
it is natural to raise the following: 
\begin{problem}
\textbf{(Mahler)}\emph{ Find the monic, integer, non-cyclotomic polynomial
with the smallest Mahler measure.}
\end{problem}
This is an 80 years old open problem of mathematics. Of course, we
can expect that the polynomials with the smallest Mahler measure be
among those with only a few number of zeros outside $\mathbb{S}$,
in particular among those with only one zero outside of $\mathbb{S}$.
A monic integer polynomial that has exactly one zero outside $\mathbb{S}$
is called a \emph{Pisot polynomial} and its unique zero of modulus
greater than $1$ is called its \emph{Pisot number} \cite{Bertin2012}.
A breakthrough towards the solution of Mahler's problem was given
by Smyth in \cite{Smyth1971}:
\begin{thm}
\textbf{\emph{(Smyth)}} \label{TheoremSmyth}The Pisot polynomial
$S(z)=z^{3}-z-1$ is the polynomial with smallest Mahler measure among
the set of all monic, integer and non-SR polynomials. Its Mahler measure
is given by the value of its Pisot number, which is, 
\begin{equation}
\sigma=\sqrt[3]{\tfrac{1}{2}+\tfrac{1}{2}\sqrt{\tfrac{23}{27}}}+\sqrt[3]{\tfrac{1}{2}-\tfrac{1}{2}\sqrt{\tfrac{23}{27}}}\approx1.32471795724.
\end{equation}
\end{thm}
The Mahler problem is, however, still open for SR polynomials. A monic
integer SR polynomial with exactly two (real and positive) zeros (say,
$\zeta$ and $1/\zeta$) not lying on $\mathbb{S}$ is called a \emph{Salem
polynomial} \cite{Bertin2012,Smyth2015}. It can be shown that a Pisot
polynomial with at least one zero on $\mathbb{S}$ is also a Salem
polynomial. The unique positive zero greater than one of a Salem polynomial
is called its \emph{Salem number}, which also equals the value of
its Mahler measure. A Salem number $s$ is said to be small if $s<\sigma$;
up to date, only $47$ small Salem numbers are known \cite{Boyd1977,Mossinghoff1998}
and the smallest known one was found about 80 years ago by Lehmer
\cite{Lehmer1933}. This gave place to the following:
\begin{conjecture}
\textbf{\emph{(Lehmer)}}\emph{ }The monic integer polynomial with
the smallest Mahler measure is the Lehmer polynomial $\mathcal{L}(z)=z^{10}+z^{9}-z^{7}-z^{6}-z^{5}-z^{4}-z^{3}+z+1$,
a Salem polynomial whose Mahler measure is $\varLambda\approx1.17628081826$,
known as Lehmer's constant. 
\end{conjecture}
The proof of this conjecture is also an open problem. To be fair,
we do not even know if there exists a smallest Salem number at all.
This is the content of another problem raised by Lehmer:
\begin{problem}
\textbf{(Lehmer)} \emph{Answer whether there exists or not a positive
number $\epsilon$ such that the Mahler measure of any monic, integer
and non-cyclotomic polynomial $p(z)$ satisfies the inequality $M[p(z)]>1+\epsilon$.}
\end{problem}
Lehmer's polynomial also appears in connection with several fields
of mathematics. Many examples are discussed in Hironaka's paper \cite{Hironaka2009};
here we shall only present an amazing identity found by Bailey \&
Broadhurst in \cite{Bailey1999} in their works on polylogarithm ladders:
if $\lambda$ is any zero of the aforementioned Lehmer's polynomial
$\mathcal{L}(z)$, then, 
\begin{equation}
\frac{\left(\lambda^{315}-1\right)\left(\lambda^{210}-1\right)\left(\lambda^{126}-1\right)^{2}\left(\lambda^{90}-1\right)\left(\lambda^{3}-1\right)^{3}\left(\lambda^{2}-1\right)^{5}\left(\lambda-1\right)^{3}}{\left(\lambda^{630}-1\right)\left(\lambda^{35}-1\right)\left(\lambda^{15}-1\right)^{2}\left(\lambda^{14}-1\right)^{2}\left(\lambda^{5}-1\right)^{6}\lambda^{68}}=1.
\end{equation}

\subsection{Knot theory}

A \emph{knot} is a closed, non-intersecting, one-dimensional curve
embedded on $\mathbb{R}^{3}$ \cite{Adams2004}. Knot theory studies
topological properties of knots as, for example, criteria under which
a knot can be unknot, conditions for the equivalency between knots,
the classification of prime knots etc. --- see \cite{Adams2004}
for the corresponding definitions. In Figure \ref{FigureKnots} we
plotted all prime knots up to six crossings.

\begin{figure}[H]
\centering\includegraphics[scale=0.35]{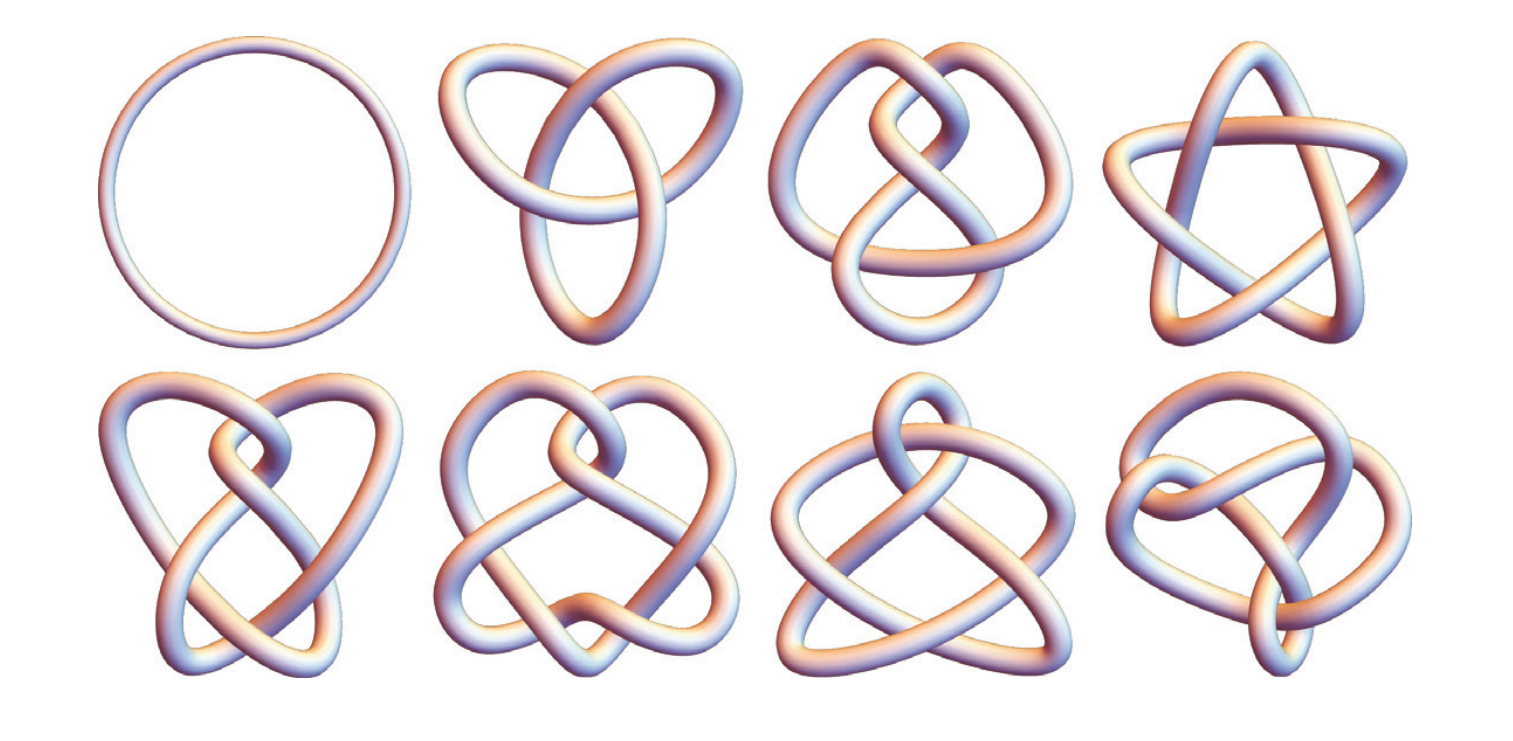}

\caption{A table of prime knots up to six crossings. In the Alexander-Briggs
notation these knots are, in order, $0_{1}$, $3_{1}$, $4_{1}$,
$5_{1}$, $5_{2}$, $6_{1}$, $6_{2}$ and $6_{3}$.}
\label{FigureKnots}
\end{figure}

One of the most important questions in knot theory is to determine
whether or not two knots are equivalent. This, however, is not an
easy task. A way of attacking this question is to look for abstract
objects --- mainly the so-called \emph{knot invariants} --- rather
than to the knots themselves. A knot invariant is a (topologic, combinatorial,
algebraic etc.) quantity that can be computed for any knot and that
is always the same for equivalents knots\footnote{We remark, however, that different knots can have the same knot invariant.
Up to date, we do not know whether there exists a knot invariant that
distinguishes all non-equivalent knots from each other (although there
do exist some invariants that distinguish every knot from the trivial
knot). Thus, until now the concept of knot invariants only partially
solves the problem of knot classification.}. An important class of knot invariants is constituted by the so-called
\emph{Knot Polynomials}. Knot polynomials were introduced in 1928
by Alexander \cite{Alexander1928}. They consist in polynomials with
integer coefficients that can be written down for every knot. For
about 60 years since its creation, Alexander polynomials were the
only known kind of knot polynomial. It was only in 1985 that Jones
\cite{Jones1985} came up with a new kind of knot polynomials ---
today known as \emph{Jones polynomials} --- and since then other
kinds were discovered as well, see \cite{Adams2004}. 

\begin{table}[H]
\centering%
\begin{tabular}{clcl}
\hline 
Knot & Alexander polynomial $\varDelta(t)$ & Knot & Alexander polynomial $\varDelta(t)$\tabularnewline
\hline 
$0_{1}$ & $1$ & $5_{2}$ & $2-3t+2t^{2}$\tabularnewline
$3_{1}$ & $1-t+t^{2}$ & $6_{1}$ & $2-5t+2t^{2}$\tabularnewline
$4_{1}$ & $1-3t+t^{2}$ & $6_{2}$ & $1-3t+3t^{2}-3t^{3}+t^{4}$\tabularnewline
$5_{1}$ & $1-t+t^{2}-t^{3}+t^{4}$ & $6_{3}$ & $1-3t+5t^{2}-3t^{3}+t^{4}$\tabularnewline
\hline 
\end{tabular}

\caption{Alexander polynomials for prime knots up to six crossings.}
\label{TableKnotP}
\end{table}

What is interesting for us here is that the Alexander polynomials
are PSR polynomials of even degree (say, $n=2m$) and with integer
coefficients\footnote{Alexander polynomials can also be defined as Laurent polynomials,
see \cite{Adams2004}.}. Thus, they have the following general form:
\begin{equation}
\varDelta(t)=\delta_{0}+\delta_{1}+\cdots+\delta_{m-1}t^{m-1}+\delta_{m}t^{m}+\delta_{m-1}t^{m+1}+\cdots+\delta_{1}t^{2m-1}+\delta_{0}t^{2m},\label{Bt}
\end{equation}
where $\delta_{i}\in\mathbb{N}$, $0\leqslant i\leqslant m$. In table
\ref{TableKnotP} we present the Alexander polynomials for the prime
knots up to six crossings. 

Knots theory finds applications in many fields of mathematics in physics
--- see \cite{Adams2004}. In mathematics, we can cite a very interesting
connection between Alexander polynomials and the theory of Salem numbers:
more precisely, the Alexander polynomial associated with the so-called
\emph{Pretzel Knot} $\mathcal{P}(-2,3,7)$ is nothing but the Lehmer
polynomial $\mathcal{L}(z)$ introduced in Section \ref{SectionSalem};
it is indeed the Alexander polynomial with the smallest Mahler measure
\cite{Hironaka2001}. In physics, knot theory is connected with quantum
groups and it also can be used to one construct solutions of the Yang-Baxter
equation \cite{Vieira2018B} through a method called baxterization
of braid groups. 

\subsection{Bethe Equations\label{SubSection-Bethe}}

The \emph{Bethe Equations} were introduced in 1931 by Hans Bethe \cite{Bethe1931}
together with his powerful method --- the so-called \emph{Bethe Ansatz
method} --- for solving spectral problems associated with exactly
integrable models of statistical mechanics. They consist in a system
of coupled and non-linear equations that ensure the consistency of
the Bethe Ansatz. In fact, for the XXZ Heisenberg spin chain, the
Bethe Equations consist in a coupled system of trigonometric equations;
however, after a change of variables is performed, we can write them
in the following rational form:
\begin{equation}
x_{i}^{L}=\left(-1\right)^{N-1}\prod_{k=1,k\neq i}^{N}\frac{x_{i}x_{k}-2\varDelta x_{i}+1}{x_{i}x_{k}-2\varDelta x_{k}+1},\qquad1\leqslant i\leqslant N,\label{Bethe}
\end{equation}
where $L\in\mathbb{N}$ is the length of the chain, $N\in\mathbb{N}$
is the excitation number and $\varDelta\in\mathbb{R}$ is the so-called
anisotropy parameter. A solution of (\ref{Bethe}) consists in a (non-ordered)
set $X=\left\{ x_{1},\ldots,x_{N}\right\} $ of the unknowns $x_{1},\ldots,x_{N}$
so that (\ref{Bethe}) is satisfied. Notice that the Bethe equations
satisfy the important relation $x_{1}^{L}x_{2}^{L}\cdots x_{N}^{L}=1$,
which suggests an inversive symmetry of their zeros. 

In \cite{VieiraLimaSantos2015}, Vieira and Lima-Santos showed that
the solutions of (\ref{Bethe}), for $N=2$ and arbitrary $L$, are
given in terms of the zeros of certain SI polynomials. In fact, (\ref{Bethe})
becomes a system of two coupled algebraic equations for $N=2$, namely,
\begin{equation}
x_{1}^{L}=-\frac{x_{1}x_{2}-2\varDelta x_{1}+1}{x_{1}x_{2}-2\varDelta x_{2}+1},\qquad\text{and}\qquad x_{2}^{L}=-\frac{x_{1}x_{2}-2\varDelta x_{2}+1}{x_{1}x_{2}-2\varDelta x_{1}+1}.\label{Bethe2}
\end{equation}
Now, from the relation $x_{1}^{L}x_{2}^{L}=1$ we can eliminate one
of the unknowns in (\ref{Bethe2}) --- for instance, by setting $x_{2}=\omega_{a}/x_{1}$,
where $\omega_{a}=\exp\left(2\pi ia/L\right)$, $1\leqslant a\leqslant L$,
are the roots of unity of degree $L$. Replacing these values for
$x_{2}$ into (\ref{Bethe2}), we obtain the following polynomial
equations fixing $x_{1}$:
\begin{equation}
p_{a}(z)=\left(1+\omega_{a}\right)z^{L}-2\varDelta\omega_{a}z^{L-1}-2\varDelta z+\left(1+\omega_{a}\right)=0,\qquad1\leqslant a\leqslant L,
\end{equation}
We can easily verify that the polynomial $p_{a}(z)$ are SI for each
value of $a$. They also satisfy the relations $p_{a}(z)=z^{L}p(\omega_{a}/z)$,
$1\leqslant a\leqslant L$, which means that the solutions of (\ref{Bethe2})
have the general form $X=\left\{ \zeta,\omega_{a}/\zeta\right\} $
for $\zeta$ any zero of $p_{a}(z)$. In \cite{VieiraLimaSantos2015}
the distribution of the zeros of the polynomials $p_{a}(z)$ was analyzed
through an application of Vieira's Theorem \ref{TheoremVieira}. It
was shown that the exact behaviour of the zeros of the polynomials
$p_{a}(z)$, for each $a$, depends on two critical values of $\varDelta$,
namely, 
\begin{equation}
\varDelta_{a}^{(1)}=\tfrac{1}{2}\left|\omega_{a}+1\right|,\qquad\text{and}\qquad\varDelta_{a}^{(2)}=\tfrac{1}{2}\left(\tfrac{L}{L-2}\right)\left|\omega_{a}+1\right|,
\end{equation}
 as follows: if $|\varDelta|\leqslant\varDelta_{a}^{(1)}$, then all
the zeros of $p_{a}(z)$ are on $\mathbb{S}$; if $|\varDelta|\geqslant\varDelta_{a}^{(2)}$,
then all the zeros of $p_{a}(z)$ but two are on $\mathbb{S}$; (see
\cite{VieiraLimaSantos2015} for the case $\varDelta_{a}^{(1)}<|\varDelta|<\varDelta_{a}^{(2)}$
and more details).

Finally, we highlight that the polynomial $p_{a}(z)$ becomes a Salem
polynomial for $a=L$ and integer values of $\varDelta$. This was
one of the first appearances of Salem polynomials in physics. 

\subsection{Orthogonal polynomials\label{Section-Orth}}

An infinite sequence $\mathcal{P}=\left\{ P_{n}(z)\right\} _{n\in\mathbb{N}}$
of polynomials $P_{n}(z)$ of degree $n$ is said to be an \emph{orthogonal
polynomial sequence} on the interval $(l,r)$ of the real line if
there exists a function $w(x)$, positive in $(l,r)\in\mathbb{R}$,
such that, 
\begin{equation}
\int_{l}^{r}P_{m}(z)P_{n}(z)w(z)dz=\begin{cases}
K_{n}, & m=n,\\
0, & m\neq n,
\end{cases}\qquad m,n\in\mathbb{N},
\end{equation}
where $K_{0}$, $K_{1}$ etc. are positive numbers. Orthogonal polynomial
sequences on the real line have many interesting and important properties
--- see \cite{Chihara2011}.

\begin{table}[H]
\centering%
\begin{tabular}{ll}
\hline 
Hermite polynomials & Möbius-transformed Hermite polynomials\tabularnewline
\hline 
$H_{0}(z)=1$ & $\mathcal{H}_{0}(z)=1$\tabularnewline
$H_{1}(z)=2z$ & $\mathcal{H}_{1}(z)=-2i-2iz$\tabularnewline
$H_{2}(z)=-2+4z^{2}$ & $\mathcal{H}_{2}(z)=-6-4z-6z^{2}$\tabularnewline
$H_{3}(z)=-12z+8z^{3}$ & $\mathcal{H}_{3}(z)=20i+12iz+12iz^{2}+20iz^{3}$\tabularnewline
$H_{4}(z)=12-48z^{2}+16z^{4}$ & $\mathcal{H}_{4}(z)=76+16z+72z^{2}+16z^{3}+76z^{4}$\tabularnewline
\hline 
\end{tabular}

\caption{Hermite and Möbius-transformed Hermite polynomials, up to $4$th degree.}
\label{TableOrth}
\end{table}

Very recently, Vieira \& Botta \cite{VieiraBotta2018A,VieiraBotta2018}
studied the action of Möbius transformations over orthogonal polynomial
sequences on the real line. In particular, they showed that the infinite
sequence $\mathcal{T}=\left\{ T_{n}(z)\right\} _{n\in\mathbb{N}}$
of the Möbius-transformed polynomials $T_{n}(z)=(z-1)^{n}P_{n}\left(W(z)\right)$,
where $W(z)=-i(z+1)/(z-1)$, is an SR and/or SI polynomial sequence
with all their zeros on the unit circle $\mathbb{S}$ --- see Table
\ref{TableOrth} for an example. We highlight that the polynomials
$T_{n}(z)\in\mathcal{T}$ also have properties similar to the original
polynomials $P_{n}(z)\in\mathcal{P}$ as, for instance, they satisfy
a type of orthogonality condition on the unit circle and a three-term
recurrence relation, their zeros lie all on $\mathbb{S}$ and are
simple, for $n\geqslant1$ the zeros of $T_{n}(z)$ interlaces with
those of $T_{n+1}(z)$ and so on --- see \cite{VieiraBotta2018A,VieiraBotta2018}
for more details.

\section{Conclusions}

In this work, we reviewed the theory of self-conjugate, self-reciprocal
and self-inversive polynomials. We discussed their main properties,
how they are related to each other, the main theorems regarding the
distribution of their zeros and some applications of these polynomials
both in physics and mathematics. We hope that this short review suits
for a compact introduction of the subject, paving the way for further
developments in this interesting field of research.

\section*{Acknowledgements}

We thanks the editorial staff for all the support during the publishing
process and also the \emph{Coordination for the Improvement of Higher
Education} (CAPES).

\bibliographystyle{elsarticle-num}
\bibliography{PolySym}

\end{document}